\documentclass[11pt]{article}%
\usepackage{amsfonts}
\usepackage{amsmath}
\usepackage{geometry}
\usepackage{amssymb}
\usepackage{graphicx}
\usepackage{hyperref}%
\usepackage{doi}
\usepackage{url}
\usepackage{color}
\providecommand{\U}[1]{\protect\rule{.1in}{.1in}}
\newtheorem{theorem}{Theorem}

\newtheorem{corollary}[theorem]{Corollary}

\newtheorem{definition}[theorem]{Definition}

\newtheorem{lemma}[theorem]{Lemma}

\newtheorem{remark}[theorem]{Remark}

\newenvironment{proof}[1][Proof]{\noindent\textbf{#1.} }{\ \rule{0.5em}{0.5em}}
\renewcommand{\d}{\;\mathrm{d}}

\begin{document}


\title{mSQG equations in distributional spaces and point vortex approximation}

\author{Franco Flandoli\thanks{Scuola Normale Superiore di Pisa 
Classe di Scienze, Pisa, Italy, e-mail: franco.flandoli@sns.it}\, and Martin 
Saal\thanks{Department of Mathematics,
TU Darmstadt, Schlossgartenstr.~7,
64289 Darmstadt, Germany, e-mail:
msaal@mathematik.tu-darmstadt.de}}

\maketitle

\begin{abstract}
Existence of distributional solutions of a modified Surface Quasi-Geostrophic 
equation (mSQG) is proved for $\mu$-almost every initial condition, where
$\mu$ is a suitable Gaussian measure. The result is the by-product of
existence of a stationary solution with white noise marginal. This solution is
constructed as a limit of random point vortices, uniformly distributed and
rescaled according to the Central Limit Theorem. 

\end{abstract}

\section{Introduction}

The Surface Quasi-Geostrophic equation (SQG) in the torus $\mathbb{T}%
^{2}=\mathbb{R}^{2}/\mathbb{Z}^{2}$ is the equation
\[%
\begin{split}
& \partial_{t}\theta+u\cdot\nabla\theta   =0,\\
& u   =\nabla^{\perp}\Lambda^{-1}\theta=\mathcal{R}^{\perp}\theta,
\end{split}
\]
where $\Lambda=(-\Delta)^{1/2}$ is the square root of the Laplacian, 
$\nabla^{\perp}=(-\partial_{y},\partial_{x})$ and $\mathcal{R}=(\mathcal{R}_1, 
\mathcal{R}_2)$ is the vector of Riesz-transforms. It has applications in both 
meteorological and oceanic flows, and it describes the temperature $\theta$ in 
a rapidly rotating stratified fluid with uniform potential vorticity (see 
\cite{HeldPierreGarnerSwanson}, \cite{Pedlosky}, \cite{Smithetal} for the 
geophysical background). In mathematics, it received a lot of attention 
because of structural similarities to the 3D Euler equations as pointed 
out in \cite{ConstantinMajdaTabak1} and \cite{ConstantinMajdaTabak2}. The 
quantity $\nabla^{\perp}$ plays a similar role for SQG as the vorticity 
does for the 3D Euler equations, and also the level sets of $\theta$ are 
analogous to vortex lines for 3D Euler equations to mention just some of 
the similarities of the systems. While for initial data $\theta_0\in H^k
(\mathbb T^2)$, the existence of unique local solutions in $C([0,T];H^k
(\mathbb{T}^2))$ for $k\geq 3$ is known since those works, the question of 
global existence of regular solutions is still open. The existence of weak 
solutions $\theta\in L^{\infty}((0,\infty);L^2(\mathbb{T}^2))$ for initial data 
in $\theta_0\in L^2(\mathbb{T}^2)$ was first established in \cite{Resnick} and 
later extended in \cite{Marchand} to the case $\theta_0\in L^p(\mathbb{T}^2)$ 
with solutions $\theta\in L^{\infty}((0,\infty);L^p(\mathbb{T}^2))$ when $4/3 < 
p<\infty$, but the uniqueness is an open problem. Note that all these results 
also hold when $\mathbb{T}^2$ is replaced by $\mathbb{R}^2$. By the technique 
of convex integration in \cite{BuckmasterShkollerVicol}, the non-uniqueness of 
weak solutions with less regularity than the above ones is established. Due to 
the properties of the nonlinearity, weak solutions to SQG can be defined even for 
distributional-valued functions $\theta\in L^2_{\text{loc}}(\mathbb{R};\dot 
H^{-1/2}(\mathbb{T}^2))$, and for sufficiently regular $\theta$, the $\dot 
H^{-1/2}$-norm is conserved, see \cite{BuckmasterShkollerVicol} and \cite{IsettVicol}
for details. However, in \cite{Staffilani2} it is shown that uniqueness can be 
restored by adding random diffusion. A solution theory in H\"older spaces is 
given in \cite{Wu}, the equation admits a local solution $\theta\in L^{\infty}
((0,T);C^r(\mathbb{R}^2)\cap L^p(\mathbb{R}^2))$ for initial data $\theta_0\in 
C^r(\mathbb{R}^2)\cap L^p(\mathbb{R}^2)$ ($r,q>1$). 

While the regularity problem for smooth
solutions is unsolved, in the recent work \cite{CastroCordobaGomez} an 
example for a non-trivial, global smooth solution with finite energy is 
constructed in the whole space setting. There are few results apart from 
the whole space $\mathbb{R}^2$ or the torus $\mathbb{T}^2$. For an open and 
bounded set $\Omega\subset \mathbb{R}^d$ ($d\geq2$) with smooth boundary in 
\cite{ConstantinNguyen}, the global existence of weak solutions  
$\theta\in L^{\infty}((0,\infty);L^2(\Omega))$ for initial data $\theta_0
\in L^2(\Omega)$ is proved. This case is more difficult because 
a helpful commutator structure in the nonlinear term is broken by the lack of 
translation invariance. 

We consider here a family of equations that links SQG to the 2D Euler equations, 
with a smoother velocity field $u$ than in SQG, called modified SQG (mSQG, 
sometimes it is also referred to as generalized SQG). It reads
\begin{equation}%
\begin{split}
& \partial_{t}\theta+u\cdot\nabla\theta =0,\\
& u   =\nabla^{\perp}\Lambda^{-1-\epsilon}\theta
\end{split}
\label{eq:msqg}%
\end{equation}
with $\epsilon\in(0,1)$. For $\epsilon=1$ we obtain the vorticity formulation of 
the 2D Euler equations and for $\epsilon=0$ SQG. For the 2D Euler equations 
the global regularity has been established (see for example \cite{MajdaBert}), 
and the above family of equations has been introduced to approach SQG 
(\cite{ChaeConstantinWu}, \cite{CordobaFontelosManchoRodrigo}). As SQG itself this 
family admits the conservation of the $L^p$-norm of solutions for $1\leq p\leq 
\infty$. For mSQG we have that similar theorems on the existence of solutions as 
for SQG hold and some additional results, including the following ones:
\begin{itemize}
\item[1.] For initial data in $\theta_0\in C^r(\mathbb{T}^2)$, there is a local 
strong solution $\theta\in L^{\infty}((0,T);C^r(\mathbb{T}^2))$ ($r>1$). The global 
existence of strong solutions is an open problem for all $\varepsilon\in(0,1]$.
\item[2.] For initial data in  $\theta_0\in L^p(\mathbb{T}^2)$, there is a global 
weak solution $\theta\in L^{\infty}((0,\infty);L^p(\mathbb{T}^2))$ ($p>4/3$). 
\item[3.] The vortex patch problem, i.e., considering as initial data the 
characteristic function of a domain with sufficiently smooth boundary (or more 
generally the sum of characteristic functions) and investigating if the boundary 
remains smooth, has a unique local solution (\cite{CordobaCordobaGancedo}, 
\cite{Gancedo}, \cite{HassainiaTaoufik}, \cite{KiselevYaoZlatos}, \cite{Rodrigo}). 
Note that here the existence of a blow-up is known in special cases 
(\cite{KiselevRyzhikYaoZlatos}). 
\item[4.]  A global flow $\psi$ defined on some probability space with values in 
$C([0,T], H^{-2-}(\mathbb{T}^2))$ for mSQG was constructed in \cite{Staffilani}. 
\end{itemize}
General domains have been studied in \cite{Nguyen} and even more singular velocity 
fields in \cite{ChaeConstantinCordobaGancedoWu}.

From a deterministic perspective, the purpose of this work is to prove
existence of solutions of class%
\[
\theta\in C\left(  \left[  0,T\right]  ;H^{-1-}\left(  \mathbb{T}^{2}\right)
\right)
\]
where $H^{-1-}\left(  \mathbb{T}^{2}\right)  =\cap_{\epsilon>0}H^{-1-\epsilon
}\left(  \mathbb{T}^{2}\right)  $. Here  $H^{s}\left(  \mathbb{T}^{2}\right)$ 
are the classical Sobolev spaces on $\mathbb{T}^{2}$ with integrability order 
$s$, defined for all real numbers $s$. Opposite to more classical existence
results in classes of more regular functions, we may prove solvability in
$C\left(  \left[  0,T\right]  ;H^{-1-}\left(  \mathbb{T}^{2}\right)  \right)
$ only for $\mu$-almost every initial condition $\theta_{0}\in H^{-1-}\left(
\mathbb{T}^{2}\right)  $, where $\mu$ is a suitable Gaussian measure supported
on $H^{-1-}\left(  \mathbb{T}^{2}\right)  $ but not on $H^{-1}\left(
\mathbb{T}^{2}\right)  $ and not on the subspace of $H^{-1-}\left(
\mathbb{T}^{2}\right)  $ of signed measures. Our result on the 
existence of such a solution is the following.
\begin{theorem}\label{theorem:existence}
For all $\epsilon\in(0,1)$, there exists a probability space $(\Xi,\mathcal{F},P)$ 
and a measurable map $\theta_{\cdot}:\Xi\times[0,T]\to H^{-1-}\left( \mathbb{T}^{2}
\right)$ ($T>0$ arbitrary) with continuous paths such that $\theta_{\cdot}$ is a 
solution to mSQG in the sense of Definition \ref{def:wnsolution} and $(\theta_{t})$ 
is stationary.
\end{theorem}
We prove only existence; uniqueness is a major open problem. Our second main result 
is the approximation of these distributional solutions by the, analogous for 
mSQG, so called point vortex solutions (we shall use this name
also for mSQG below)%
\[
\theta\left(  t\right)  =\frac{1}{\sqrt{N}}\sum_{i=1}^{N}\xi_{i}\delta
_{x_{t}^{i}}%
\]
where the dynamics of the point vortices $x_{t}^{i}$ is given by
\[
\frac{dx_{i}\left(  t\right)  }{dt}=\frac{1}{\sqrt{N}}\sum_{j=1}^{N}\xi
_{j}K\left(  x_{i}\left(  t\right)  -x_{j}\left(  t\right)  \right)  ,\qquad
i=1,\dots,N
\]
with the kernel $K$ and the intensities $\xi_{i}$ specified below. We show 
that the solution obtained in Theorem \ref{theorem:existence} can be approximated 
by the corresponding point vortex system.
\begin{theorem}\label{theorem:approximation}
The random point vortex system defined on $(\Xi,\mathcal{F},P)$
has a subsequence which converges P-a.s. to $\theta$ in $C([0,T];H^{-1-}%
(\mathbb{T}^{2}))$.
\end{theorem}
In some cases it has been proved that point vortices are limits of regular
$\theta$ -patches (for Euler equations see \cite{Flandoli}, for a class of 
mSQG models see \cite{GeldRomitoPrep}), when this occurs it implies that the 
distributional solutions constructed here are also limits of regular patches.

From a stochastic perspective, we prove existence of a stationary solution of
mSQG having marginals with the white noise law, hence supported on
$H^{-1-}\left(  \mathbb{T}^{2}\right) $. The stationary solution is
constructed as a weak limit of random point vortices, distributed uniformly in
space, with CLT-scaling of intensities. 

The results proved here for mSQG model are parallel to those obtained in
\cite{Flandoli} for 2D Euler equations. Limited to the existence of a
stationary solutions with white noise marginal, there is intersection with the
results of \cite{Staffilani}, proved by a different method, along the lines of
\cite{AlbCru}\ (based on Galerkin approximation instead of point vortices,
which makes less clear the connection with other classes of solutions to the
equation). Results on mSQG model and point vortices, in directions different
from ours, have been recently proved by \cite{GeldRomito}, \cite{GeldRomitoPrep}. 
For the whole space setting the idea of point vortices for mSQG is sketched in 
\cite{CavallaroGarraMarchioro}.

Last we want to mention that there is a second way to introduce some smoothing 
into the equation by adding viscosity of the form $\Lambda^{\gamma}\theta$ to 
the evolution equation for $\theta$ with $\gamma>0$. This dissipative SQG has 
been investigated by many authors; we refer here to \cite{Constantin} 
and the references therein for more information.

\section{The nonlinear term in the weak formulation}

\subsection{White noise}

We denote by $\left\{  e_{n}\right\}  $ the complete orthonormal system in
$L^{2}\left(  \mathbb{T}^{2},\mathbb{C}\right)  $ given by $e_{n}\left(
x\right)  =e^{2\pi in\cdot x}$, $n\in\mathbb{Z}^{2}$. Given a distribution
$\omega\in C^{\infty}\left(  \mathbb{T}^{2}\right)  ^{\prime}$ and a test
function $\phi\in C^{\infty}\left(  \mathbb{T}^{2}\right)  $, we denote by
$\left\langle \omega,\phi\right\rangle $ the duality between $\omega$ and
$\phi$ (namely $\omega\left(  \phi\right)  $), and we use the same symbol for
the inner product of $L^{2}\left(  \mathbb{T}^{2}\right)  $. We set
$\widehat{\omega}\left(  n\right)  =\left\langle \omega,e_{n}\right\rangle $,
$n\in\mathbb{Z}^{2}$, and we define, for each $s\in\mathbb{R}$, the space
$H^{s}\left(  \mathbb{T}^{2}\right)  $ as the space of all distributions
$\omega\in C^{\infty}\left(  \mathbb{T}^{2}\right)  ^{\prime}$ such that
\[
\left\Vert \omega\right\Vert _{H^{s}}^{2}:=\sum_{n\in\mathbb{Z}^{2}}\left(
1+\left\vert n\right\vert ^{2}\right)  ^{s}\left\vert \widehat{\omega}\left(
n\right)  \right\vert ^{2}<\infty.
\]
We use similar definitions and notations for the space $H^{s}\left(
\mathbb{T}^{2},\mathbb{C}\right)  $ of complex valued functions. In the space
$H^{-1-}\left(  \mathbb{T}^{2}\right)  =%
{\displaystyle\bigcap\limits_{\epsilon>0}}
H^{-1-\epsilon}\left(  \mathbb{T}^{2}\right)  $, we consider the metric%
\[
d_{H^{-1-}}\left(  \omega,\omega^{\prime}\right)  =\sum_{n=1}^{\infty}%
2^{-n}\left(  \left\Vert \omega-\omega^{\prime}\right\Vert _{H^{-1-\frac{1}%
{n}}}\wedge1\right)  .
\]
Convergence in this metric is equivalent to convergence in $H^{-1-\epsilon
}\left(  \mathbb{T}^{2}\right)  $ for every $\epsilon>0$. The space
$H^{-1-}\left(  \mathbb{T}^{2}\right)  $ with this metric is complete and
separable. We denote by $\mathcal{X}:=C\left(  \left[  0,T\right]
;H^{-1-}\left(  \mathbb{T}^{2}\right)  \right)  $ the space of continuous
functions with values in this metric space;\ a function is in $\mathcal{X}$ if
and only if it is in $C\left(  \left[  0,T\right]  ;H^{-1-\epsilon}\left(
\mathbb{T}^{2}\right)  \right)  $ for every $\epsilon>0$. The distance in
$C\left(  \left[  0,T\right]  ;H^{-1-}\left(  \mathbb{T}^{2}\right)  \right)
$ is given by $d_{\mathcal{X}}\left(  \omega_{\cdot},\omega_{\cdot}^{\prime
}\right)  =\sup_{t\in\left[  0,T\right]  }d_{H^{-1-}}\left(  \omega_{t}%
,\omega_{t}^{\prime}\right)  $, which makes $\mathcal{X}$ a Polish space.

For $s>0$, the spaces $H^{s}\left(  \mathbb{T}^{2}\right)  $ and
$H^{-s}\left(  \mathbb{T}^{2}\right)  $ are dual each other. By $H^{s+}\left(
\mathbb{T}^{2}\right)  $, we shall therefore mean the space $%
{\displaystyle\bigcup\limits_{\epsilon>0}}
H^{s+\epsilon}\left(  \mathbb{T}^{2}\right)  $. We will use this notation also 
in the case of the space $H^{2+}\left(  \mathbb{T}^{2}\times\mathbb{T}%
^{2}\right)  $, which is similarly defined.

We start recalling the well known notion of white noise, reviewing some of its
main properties used in the sequel.

White noise on $\mathbb{T}^{2}$ is by definition a Gaussian
distributional-valued stochastic process $\omega:\Xi\rightarrow C^{\infty
}\left(  \mathbb{T}^{2}\right)  ^{\prime}$, defined on some probability space
$\left(  \Xi,\mathcal{F},\mathbb{P}\right)  $, such that
\begin{equation}
\mathbb{E}\left[  \left\langle \omega,\phi\right\rangle \left\langle
\omega,\psi\right\rangle \right]  =\left\langle \phi,\psi\right\rangle
\label{def WN}%
\end{equation}
for all $\phi,\psi\in C^{\infty}\left(  \mathbb{T}^{2}\right)  $ (Gaussian
means that the real-valued r.v. $\left\langle \omega,\phi\right\rangle $ is
Gaussian, for every $\phi\in C^{\infty}\left(  \mathbb{T}^{2}\right)  $). We
have denoted by $\left\langle \omega\left(  \theta\right)  ,\phi\right\rangle
$ the duality between the distribution $\omega\left(  \theta\right)  $ (for
some $\theta\in\Xi$)\ and the test function $\phi\in C^{\infty}\left(
\mathbb{T}^{2}\right)  $. These properties uniquely characterize the law of
$\omega$. In more heuristic terms, as it is often written in the physics
literature,%
\[
\mathbb{E}\left[  \omega\left(  x\right)  \omega\left(  y\right)  \right]
=\delta\left(  x-y\right)
\]
since double integration of this identity against $\phi\left(  x\right)
\psi\left(  y\right)  $ gives (\ref{def WN}). White noise exists: it is
sufficient to take the complete orthonormal system $\left\{  e_{n}\right\}
_{n\in\mathbb{Z}^{2}}$ of $L^{2}\left(  \mathbb{T}^{2},\mathbb{C}\right)  $
as above, a probability space $\left(\Xi,\mathcal{F},\mathbb{P}\right)  $ 
supporting a sequence of independent standard Gaussian variables 
$\left\{  G_{n}\right\}  _{n\in\mathbb{Z}^{2}}$, and consider the series%
\[
\omega=\sqrt{2}\operatorname{Re}\sum_{n\in\mathbb{Z}^{2}}G_{n}e_{n}\text{.}%
\]
The partial sums $\omega_{N}^{\mathbb{C}}\left(  \theta,x\right)
=\sum_{\left\vert n\right\vert \leq N}G_{n}\left(  \theta\right)  e_{n}\left(
x\right)  $ are well defined complex valued random fields with square
integrable paths, $\omega_{N}:\Xi\rightarrow L^{2}\left(  \mathbb{T}%
^{2},\mathbb{C}\right)  $. For every $\epsilon>0$, $\left\{  \omega
_{N}^{\mathbb{C}}\right\}  _{N\in\mathbb{N}}$ is a Cauchy sequence in
$L^{2}\left(  \Xi;H^{-1-\epsilon}\left(  \mathbb{T}^{2},\mathbb{C}\right)
\right)  $, because
\begin{multline*}
\mathbb{E}\left[  \left\Vert \omega_{N}^{\mathbb{C}}\left(  \theta,x\right)
-\omega_{M}^{\mathbb{C}}\left(  \theta,x\right)  \right\Vert _{H^{-1-\epsilon
}}^{2}\right]  =\mathbb{E}\left[  \sum_{M<\left\vert n\right\vert \leq
N}\left(  1+\left\vert n\right\vert ^{2}\right)  ^{-1-\epsilon}\left\vert
G_{n}\right\vert ^{2}\right] \\
=\sum_{M<\left\vert n\right\vert \leq N}\left(
1+\left\vert n\right\vert ^{2}\right)  ^{-1-\epsilon}.
\end{multline*}
The limit $\omega^{\mathbb{C}}$ in $L^{2}\left(  \Xi;H^{-1-\epsilon}\left(
\mathbb{T}^{2},\mathbb{C}\right)  \right)  $ thus exists, and $\omega=\sqrt
{2}\operatorname{Re}\omega^{\mathbb{C}}$ is a white noise because (doing
rigorously the computation on the partial sums and then taking the limit) it
is centered and for $\phi,\psi\in C^{\infty}\left(  \mathbb{T}^{2}\right)  $,
\begin{align*}
\mathbb{E}\left[  \left\langle \omega,\phi\right\rangle \left\langle
\omega,\psi\right\rangle \right]   &  =\operatorname{Re}\mathbb{E}\left[
\left\langle \omega_{C},\phi\right\rangle \left\langle \overline{\omega_{C}%
},\psi\right\rangle \right]  =\operatorname{Re}\sum_{n,m\in\mathbb{Z}^{2}%
}\left\langle e_{n},\phi\right\rangle \overline{\left\langle e_{m}%
,\psi\right\rangle }\mathbb{E}\left[  G_{n}G_{m}\right] \\
&  =\operatorname{Re}\sum_{n\in\mathbb{Z}^{2}}\left\langle e_{n}%
,\phi\right\rangle \overline{\left\langle e_{n},\psi\right\rangle
}=\left\langle \phi,\psi\right\rangle .
\end{align*}
[One obtains the same result by taking $\omega=\sum_{n\in\mathbb{Z}^{2}}%
G_{n}e_{n}$ where $\mathbb{Z}^{2}\backslash\left\{  0\right\}  $ is
partitioned as $\mathbb{Z}^{2}=\Lambda\cup\left(  -\Lambda\right)  $; $G_{n} $
are i.i.d. $N\left(  0,1\right)  $ on $\Lambda\cup\left\{  0\right\}  $ and
$G_{-n}=\overline{G_{n}}$ for $n\in\Lambda$.] The law $\mu$ of the measurable
map $\omega:\Xi\rightarrow H^{-1-\epsilon}\left(  \mathbb{T}^{2}\right)  $ is
a Gaussian measure (it is sufficient to check that $\left\langle \omega
,\phi\right\rangle $ is Gaussian for every $\phi\in C^{\infty}\left(
\mathbb{T}^{2}\right)  $, and this is true since $\left\langle \omega
,\phi\right\rangle $ is the $L^{2}\left(  \Xi\right)  $-limit of the Gaussian
variables $\sum_{\left\vert n\right\vert \leq N}G_{n}\left\langle e_{n}%
,\phi\right\rangle $). The measure $\mu$ is supported by $H^{-1-}\left(
\mathbb{T}^{2}\right)  $ but not by $H^{-1}\left(  \mathbb{T}^{2}\right)  $;
namely, we have%
\[
\mu\left(  H^{-1}\left(  \mathbb{T}^{2}\right)  \right)  =0.
\]
This follows from
\[
\mathbb{E}\left[  \left\Vert \omega^{\mathbb{C}}\right\Vert _{H^{-1}}%
^{2}\right]  =\sum_{n\in\mathbb{Z}^{2}}\left(  1+\left\vert n\right\vert
^{2}\right)  ^{-1}=+\infty.
\]

The measure $\mu$ is sometimes denoted heuristically as%
\[
\mu\left(  d\omega\right)  =\frac{1}{Z}\exp\left(  -\frac{1}{2}\int%
_{\mathbb{T}^{2}}\omega^{2}dx\right)  d\omega
\]
and called the \textit{enstrophy measure}.\\
The notation "$d\omega$" has no
meaning (unless interpreted as a limit of measures on finite dimensional
Euclidean spaces), it just reminds the structure of centered nonsingular Gaussian
measures in $\mathbb{R}^{n}$, that is $\mu_{n}\left(  d\omega_{n}\right)
=\frac{1}{Z_{n}}\exp\left(  -\frac{1}{2}\left\langle Q_{n}^{-1}\omega
_{n},\omega_{n}\right\rangle _{\mathbb{R}^{n}}\right)  d\omega_{n}$ where
$d\omega_{n}$ is Lebesgue measure in $\mathbb{R}^{n}$ and $Q_{n}$ is the
covariance matrix. The notation $\int_{\mathbb{T}^{2}}\omega^{2}dx$ alludes to
the fact that $\mu$, heuristically considered as a Gaussian measure on
$L^{2}\left(  \mathbb{T}^{2}\right)  $ (this is not possible, $\mu\left(
L^{2}\left(  \mathbb{T}^{2}\right)  \right)  =0$), has covariance equal to the
identity: if $Q=Id$, then $\left\langle Q^{-1}\omega,\omega\right\rangle
_{L^{2}}=\int_{\mathbb{T}^{2}}\omega^{2}dx$. The fact that in $L^{2}\left(
\mathbb{T}^{2}\right)  $ the covariance operator $Q$, heuristically defined as%
\[
\left\langle Q\omega,\omega\right\rangle _{L^{2}}=\mathbb{E}\left[
\left\langle \omega,\phi\right\rangle _{L^{2}}\left\langle \omega
,\psi\right\rangle _{L^{2}}\right]
\]
is the identity in the case of the law $\mu$ of white noise, is a simple
"consequence" (the argument is not rigorous ab initio) of the definition
(\ref{def WN}) of white noise.

\subsection{Weak formulation of mSQG}

Before we can turn to the construction of a white noise solution, we need to
fix our notion of solution in such a weak setting; especially, we need to find
a good interpretation of the nonlinear term. To this end, we integrate
(\ref{eq:msqg}) in time. Using that $u$ is divergence free, we get for any
test-function $\phi\in C^{\infty}(\mathbb{T}^{2})$
\begin{align*}
\left<  \theta_{t},\phi\right>  =\left<  \theta_{0},\phi\right>  + \int%
_{0}^{t} \left<  \theta_{s} u_{s},\nabla\phi\right>  \d s
\end{align*}
From $u=\nabla^{\perp}\Lambda^{-1-\epsilon}\theta$, it follows
\begin{align*}
u_{t}(x)=\int_{\mathbb{T}^{2}} K(x-y) \theta_{t}(y) \d y
\end{align*}
where the kernel $K$ is in the whole space given by $K(x)=c \frac{x^{\perp}%
}{|x|^{3-\epsilon}}$ for some constant $c>0$. In the torus, we have $K$ smooth 
for $x\neq0$, $K(x-y)=-K(y-x)$ and $K(x)\leq\frac{c}{|x|^{2-\epsilon}}$ for 
$|x|$ small. We set $K(0)=0$. By the symmetry of the kernel, we obtain the 
Schochet symmetrisation
\begin{align*}
\int_{0}^{t} \left<  \theta_{s} u_{s},\nabla\phi\right>  \d s  &  = \int%
_{0}^{t} \int_{\mathbb{T}^{2}}\int_{\mathbb{T}^{2}}\theta_{s}(x) K(x-y)
\theta_{s}(y) \nabla\phi(x) \d y \d x \d s\\
&  = \frac12 \int_{0}^{t}\int_{\mathbb{T}^{2}}\int_{\mathbb{T}^{2}}\theta
_{s}(x) K(x-y) \theta_{s}(y) \nabla(\phi(x)-\phi(y)) \d y \d x \d s\\
&  = \frac12 \int_{0}^{t}\int_{\mathbb{T}^{2}}\int_{\mathbb{T}^{2}}\theta
_{s}(x) \theta_{s}(y) H_{\phi}(x,y) \d y \d x \d s
\end{align*}
where we set $H_{\phi}(x,y) = K(x-y)\nabla(\phi(x)-\phi(y))$. Such a 
symmetrisation was carried out in \cite{Schochet} for the 2D Euler equations 
and it is in the same spirit as rewriting the nonlinearity with the commutator 
used for example in \cite{BuckmasterShkollerVicol}, \cite{Marchand} and 
\cite{Resnick}. Due to the crucial role of $H_{\phi}$ throughout this work 
let us give a rigorous definition and collect some properties of it.

\begin{definition}
Let $K:\mathbb{T}^{2}\to\mathbb{R}^{2}$ be smooth for $x\neq 0$, $K(x)=-K(-x)$ 
for all $x\in\mathbb{T}^{2}$, $K(0)=0$, $K(x)\leq\frac{c}{|x|^{2-\epsilon}}$ for 
$|x|$ small and $\phi\in C^{\infty}(\mathbb{T}^{2})$. Then the function
\begin{align*}
H_{\phi}:\mathbb{T}^{2}\times\mathbb{T}^{2}\to\mathbb{R}, \; H_{\phi}(x,y) =
K(x-y)\nabla(\phi(x)-\phi(y))
\end{align*}
is in $L^{2}(\mathbb{T}^{2}\times\mathbb{T}^{2})$, $H_{\phi}(y,x)=H_{\phi
}(x,y)$ for all $x,y\in\mathbb{T}^{2}$, $H_{\phi}(x,x)=0$ and $H_{\phi}$ is
smooth outside of the diagonal.
\end{definition}

Our next purpose is to define
\begin{align*}
\int_{\mathbb{T}^{2}}\int_{\mathbb{T}^{2}}\omega(x) \omega(y) H_{\phi}(x,y) \d
y \d x \d s
\end{align*}
when $\omega:\Xi\to C^{\infty}(\mathbb{T}^{2})^{\prime}$ is a white noise.

If $\omega\in C^{\infty}(\mathbb{T}^{2})^{\prime}$ is a distribution, we can
define a distribution $\omega\otimes\omega\in C^{\infty}(\mathbb{T}^{2}
\times\mathbb{T}^{2})^{\prime}$ by setting
\begin{align*}
\left<  \omega\otimes\omega, \phi\otimes\psi\right>  = \left<  \omega,
\phi\right>  \left<  \omega, \psi\right>  \qquad(\phi,\psi\in C^{\infty
}(\mathbb{T}^{2}))
\end{align*}
where $(\phi\otimes\psi)(x,y)=\phi(x)\psi(y)$ and extend this to all $f\in
C^{\infty}(\mathbb{T}^{2} \times\mathbb{T}^{2})$ by density arguments. If
$\omega\in H^{-s}(\mathbb{T}^{2})$ for some $s>0$, then $\omega\otimes\omega\in
H^{-2s}(\mathbb{T}^{2}\times\mathbb{T}^{2})$. In the case of $\omega$ being a
white noise, we have the following result for this distribution (taken from 
\cite{Flandoli}).

\begin{corollary}
Let $\omega:\Xi\to C^{\infty}(\mathbb{T}^{2})^{\prime}$ be a white noise and
$f\in H^{2+}(\mathbb{T}^{2}\times\mathbb{T}^{2})$.

\begin{itemize}
\item[i)] We have $\mathbb{E}[\left<  \omega\otimes\omega,f\right>  ]
=\int_{\mathbb{T}^{2}}f(x,x)\d x$.

\item[ii)] If $f$ is symmetric, we have
\begin{align*}
\mathbb{E}\left[  |\left<  \omega\otimes\omega,f\right>  - \mathbb{E}[\left<
\omega\otimes\omega,f\right>  ]|^{2} \right]  =\int_{\mathbb{T}^{2}}%
\int_{\mathbb{T}^{2}}f^{2}(x,y)\d x \d y.
\end{align*}

\end{itemize}
\end{corollary}

Although our function $H_{\phi}$ is not of class $H^{2+} (\mathbb{T}^{2}%
\times\mathbb{T}^{2})$, we can use this corollary to give a meaning to $\left<
\omega\otimes\omega,H_{\phi}\right>  $ in $L^{2}(\Xi)$ by a suitable
approximation. An obvious consequence of ii) is the following.

\begin{corollary}
Let $\omega:\Xi\to C^{\infty}(\mathbb{T}^{2})^{\prime}$ be a white noise and
$\phi\in C^{\infty}(\mathbb{T}^{2})$ be given. Furthermore, let $(H_{\phi}%
^{n})_{n} \in H^{2+}(\mathbb{T}^{2}\times\mathbb{T}^{2})$ be a sequence of
symmetric functions such that
\begin{align*}
\|H_{\phi}-H_{\phi}^{n}\|_{L^{2}(\mathbb{T}^{2}\times\mathbb{T}^{2})}  &  \to0
\quad(n\to\infty)
\end{align*}
Then $\left<  \omega\otimes\omega, H_{\phi}^{n} \right>  -\mathbb{E}%
[\omega\otimes\omega, H_{\phi}^{n}]$ is a Cauchy sequence in $L^{2}(\Xi)$.
\end{corollary}

Choosing the sequence $H_{\phi}^{n}$ a bit more specific, we obtain the
following convergence result. More precisely, we take a sequence such that
$\mathbb{E} [\omega\otimes\omega, H_{\phi}^{n}]$ vanishes in the limit, and
thus, we somehow normalize our limit to have expectation $0$.

\begin{corollary}
\label{cor:defhphi} Let $\omega$, $\phi$ and $(H_{\phi}^{n})_{n}$ be as 
above and assume additionally
\begin{align*}
\int_{\mathbb{T}^{2}} H_{\phi}^{n}(x,x) \d x \to0 \quad(n\to\infty).
\end{align*}
Then, $\left<  \omega\otimes\omega, H_{\phi}^{n} \right>  $ is a Cauchy
sequence in $L^{2}(\Xi)$ and we denote its limit by $\left<  \omega
\otimes\omega, H_{\phi} \right>  $. This limit does not depend on the specific
choice of $(H_{\phi}^{n})_{n}$, and we have $\mathbb{E}[|\left<  \omega
\otimes\omega,H_{\phi}\right>  |^{2}]= \|H_{\phi}\|_{L^{2}(\mathbb{T}%
^{2}\times\mathbb{T}^{2})}$.
\end{corollary}

\begin{proof}
By the above corollary, we have that $\left<  \omega\otimes\omega, H_{\phi}^{n}
\right>  -\mathbb{E}[\left<  \omega\otimes\omega, H_{\phi}^{n}\right>  ]$
converges in $L^{2}(\Xi)$, and by our assumption, $\mathbb{E}[\omega
\otimes\omega, H_{\phi}^{n}]= \int_{\mathbb{T}^{2}} H_{\phi}^{n}(x,x) \d x
\to0 \quad(n\to\infty)$, so $\left<  \omega\otimes\omega, H_{\phi}^{n}
\right>  $ converges in $L^{2}(\Xi)$. For the estimate, note that
\begin{align*}
\mathbb{E}[|\left<  \omega\otimes\omega, H_{\phi}\right>  |^{2}]  &
=\lim_{n\to\infty}\mathbb{E}[|\left<  \omega\otimes\omega,H^{n}_{\phi}\right>
|^{2}]\\
&  =\lim_{n\to\infty} \mathbb{E}[|\left<  \omega\otimes\omega, H^{n}_{\phi
}\right>  - \mathbb{E}[\left<  \omega\otimes\omega, H_{\phi}^{n}\right>
]|^{2}]\\
&  =\lim_{n\to\infty} \|H_{\phi}^{n}\|^{2}_{L^{2}(\mathbb{T}^{2}%
\times\mathbb{T}^{2})} =\|H_{\phi}\|^{2}_{L^{2}(\mathbb{T}^{2}\times
\mathbb{T}^{2})}.
\end{align*}

\end{proof}

\begin{remark}
Constructing such a sequence is straight forward by cutting out the diagonal
with some smooth cut-off function: Let $\varphi$ be smooth, symmetric, with
$\mathrm{supp}\; \varphi\in B(0,1)$. $0\leq\varphi\leq1$, $\varphi\equiv1$ in
$B(0,1/2)$. We define $\varphi_{n}:=\varphi(n x)$ and set
\begin{align*}
H_{\phi}^{n}(x,y):=(1-\varphi_{n}(x-y))H_{\phi}(x,y).
\end{align*}
By definition we have $H_{\phi}^{n}(x,x)=0$ and
\begin{align*}
\|H_{\phi}-H_{\phi}^{n}\|_{L^{2}(\mathbb{T}^{2}\times\mathbb{T}^{2})}^{2}  &
=\int_{\mathbb{T}^{2}}\int_{\mathbb{T}^{2}}|H_{\phi}(x,y)|^{2}\varphi_{n}%
^{2}(x-y)\d x\d y \leq\int_{\mathbb{T}^{2}} \int\limits_{|x-y|\leq\frac1n}
|H_{\phi}(x,y)|^{2} \d x \d y\\
&  \leq\int_{\mathbb{T}^{2}}\int\limits_{|x-y|\leq\frac1n} |x-y|^{-2+2\epsilon
} \d x \d y\\
&  \leq c \frac{1}{n^{2\epsilon}}%
\end{align*}
so $H_{\phi}^{n}$ has the desired properties to apply the previous corollary.

For $\epsilon=0$, we have that $H_{\phi}\notin L^{2}(\mathbb{T}^{2}%
\times\mathbb{T}^{2})$ but only $H_{\phi}\in L^{p}(\mathbb{T}^{2}%
\times\mathbb{T}^{2})$ for $1\leq p<2$, and thus, the construction here does not
work ($\|H_{\phi}-H_{\phi}^{n} \|_{L^{2}(\mathbb{T}^{2}\times\mathbb{T}^{2})}$
diverges logarithmically).
\end{remark}

The above results also hold when we have an additional time dependence, and
thus, they give rise to a good notion of a weak solution.

\begin{theorem}
Let $\theta:\Xi\times[0,T]\to C^{\infty}(\mathbb{T}^{2})^{\prime}$ be
measurable with $\theta(\xi)\in C([0,T];H^{-1-}(\mathbb{T}^{2}))$. Assume that
$\theta_{t}$ is a white noise for all $t\in[0,T]$. For $\phi\in C^{\infty
}(\mathbb{T}^{2})$ let $H^{n}_{\phi}$ be an approximation of $H_{\phi}$ as in
Corollary \ref{cor:defhphi}. Then, $\left<  \theta_{\cdot} \otimes\theta
_{\cdot}, H_{\phi}^{n} \right>  $ is a Cauchy sequence in $L^{2}(\Xi;
L^{2}(0,T))$ and we denote the limit by $\left<  \theta_{\cdot} \otimes
\theta_{\cdot}, H_{\phi}\right>  $.
\end{theorem}

We now have everything in place to define what we mean by a white noise
solution of mSQG.

\begin{definition}
\label{def:wnsolution} Let $\theta:\Xi\times[0,T]\to C^{\infty}(\mathbb{T}%
^{2})^{\prime}$ be measurable with $\theta(\xi)\in C([0,T];H^{-1-}%
(\mathbb{T}^{2}))$ for almost all $\xi\in\Xi$. We call $\theta$ a white noise solution
of mSQG if $\theta_{t}$ is a white noise for all $t\in[0,T]$ and
\begin{align*}
\left<  \theta_{t},\phi\right>  = \left<  \theta_{0},\phi\right>  +\int%
_{0}^{t} \left<  \theta_{s} \otimes\theta_{s}, H_{\phi} \right>  \;
\mathrm{d}s
\end{align*}
$P$-a.s. for every $\phi\in C^{\infty}(\mathbb{T}^{2})$.
\end{definition}

\section{Random point vortex dynamics}

The first step to construct a white noise solution is to consider a finite
number of particles whose interaction is described by the kernel $K$. We
define
\begin{align*}
\Delta_{N}:= \left\{  (x_{1},\dots,x_{N})\in(\mathbb{T}^{2})^{N}:x_{i}=x_{j}
\text{ for some } i \neq j, \; i,j=1,\dots,N \right\}  .
\end{align*}
The dynamics of point vortices are given by the finite dimensional system
\begin{align}
\label{eq:odevortexdynamics}\frac{dx_{i}\left(  t\right)  }{dt}=\frac{1}%
{\sqrt{N}}\sum_{j=1}^{N}\xi_{j} K\left(  x_{i}\left(  t\right)  -x_{j}\left(
t\right)  \right)  ,\qquad i=1,\dots,N
\end{align}
where $\xi_{j}\in\mathbb{R}$ is the intensity of the vortex $x_{j}\left(
t\right)  $ and the initial value $(x_{1}(0),\dots,x_{N}(0))\in\Delta_{N}^{c}$
is given. Due to $K(0)=0$ the term with $i=j$ in the sum vanishes. First we
show that for a.e. $(x_{1}(0),\dots,x_{N}(0))\in\Delta_{N}^{c}$ the
point vortex dynamics are globally well-defined, where we consider the Lebesgue 
measure on the torus. This is obviously the case if there is no collision, i.e., 
if $x(t):=(x_{1}(t),\dots,x_{N}(t)\in\Delta_{N}^{c})$ for all $t$.

\subsection{No collision for a.e. initial condition}

We present a variant of a famous result of solvability for a.e. initial
conditions, taken from \cite{DurrPulv} (on the torus) and \cite{MarchPulv
book}, Chapter 4 (in full space). We work on the unitary torus $\mathbb{T}%
^{2}=\mathbb{R}^{2}/\mathbb{Z}^{2}$, as in \cite{DurrPulv}; here Lebesgue
measure is a probability, and particle displacement is controlled a priori by
the compactness of the set. Recall that the singular part of interaction
kernel $K\left(  x-y\right)  $ on $\mathbb{T}^{2}$ is locally of the form%
\[
K\left(  x-y\right)  \sim\frac{\left(  x-y\right)  ^{\perp}}{\left\vert
x-y\right\vert ^{3-\epsilon}}%
\]
for some $\epsilon>0$. The function $K$ is given by the orthogonal gradient of
a certain Green-type function $G$:%
\[
K\left(  x\right)  =\nabla^{\perp}G\left(  x\right)
\]
and the singular part of $G\left(  x\right)  $ is locally of the form (up to
multiplicative constants)
\[
G\left(  x-y\right)  \sim\frac{1}{\left\vert x-y\right\vert ^{1-\epsilon}}.
\]
We introduce the (Lyapunov)\ function%
\[
L\left(  x_{1},\dots,x_{N}\right)  =-\sum_{\substack{i,j=1,\dots,N\\i\neq
j}}G\left(  x_{i}-x_{j}\right)  .
\]
By a simple argument, if we prove for a given an initial condition
\[
L\left(  x_{1}\left(  t\right)  ,\dots,x_{N}\left(  t\right)  \right)  \leq C
\]
then particles do not collapse and the dynamics is globally defined for that
initial condition, because $G$ is strictly negative in a neighborhood of the 
singularity.

With $L\left(  t\right)  =L\left(  x_{1}\left(  t\right)  ,\dots,x_{N}\left(
t\right)  \right)  $, we have%
\begin{align*}
\frac{dL\left(  t\right)  }{dt}  &  =-\sum_{\substack{i,j=1,\dots,N\\i\neq
j}}\nabla G\left(  x_{i}\left(  t\right)  -x_{j}\left(  t\right)  \right)
\left(  \frac{dx_{i}\left(  t\right)  }{dt}-\frac{dx_{j}\left(  t\right)
}{dt}\right) \\
&  =-\sum_{\substack{i,j=1,\dots,N\\i\neq j}}\nabla G\left(  x_{i}\left(
t\right)  -x_{j}\left(  t\right)  \right) \\
&  \qquad\qquad\qquad\cdot\frac{1}{\sqrt{N}}\left(\sum_{k\neq i}\xi_{k}\nabla
^{\perp}G\left( x_{i}\left( t\right)-x_{k}\left( t\right)\right)-\sum_{k\neq
j}\xi_{k}\nabla^{\perp}G\left(  x_{j}\left(  t\right)  -x_{k}\left(  t\right)
\right)  \right)  .
\end{align*}
Here we see an important cancellation (its importance will be appreciated
below): the term in the sum $\sum_{k\neq i}$ with $k=j$ and the term in the
sum $\sum_{k\neq j}$ with $k=i$ do not contribute, because%
\[
\nabla G\left(  x_{i}\left(  t\right)  -x_{j}\left(  t\right)  \right)
\cdot\nabla^{\perp}G\left(  x_{i}\left(  t\right)  -x_{j}\left(  t\right)
\right)  =0.
\]
These are the most singular terms, since for small $\left\vert x_{i}\left(
t\right)  -x_{j}\left(  t\right)  \right\vert $ they behave like
\[
\frac{1}{\left\vert x_{i}\left(  t\right)  -x_{j}\left(  t\right)  \right\vert
^{4-2\epsilon}}.
\]
The other terms behave like%
\[
\frac{1}{\left\vert x_{i}\left(  t\right)  -x_{j}\left(  t\right)  \right\vert
^{2-\epsilon}}\frac{1}{\left\vert x_{i}\left(  t\right)  -x_{k}\left(
t\right)  \right\vert ^{2-\epsilon}}%
\]
with $j\neq k$; hence, they are less singular when two particles approach each 
other.

In order to make progresses, we now need to consider the flow map
$x^{0}\mapsto x\left(  t|x^{0}\right)  $. This is locally defined, when
$x^{0}\in\Delta_{N}^{c}$. However, the time before collision depends on
$x^{0}$ and complicate matters. To avoid these troubles, we mollify $G$ in
such a way that we have global solutions for all $x^{0}$, a smooth flow, but
also equal to the original solutions if particles are not too close each other.

For every $\delta\in\left(  0,1\right)  $, denote by $G^{\left(
\delta\right)  }\left(  x\right)  $ a smooth periodic function such that
\begin{align*}
G^{\left(  \delta\right)  }\left(  x\right)   &  =G\left(  x\right)  \text{
for }\left\vert x\right\vert \geq\delta\\
\left\vert G^{\left(  \delta\right)  }\left(  x\right)  \right\vert  &
\leq\frac{C}{\left\vert x-y\right\vert ^{1-\epsilon}}\text{ for }\left\vert
x\right\vert >0\\
\left\vert \nabla G^{\left(  \delta\right)  }\left(  x\right)  \right\vert  &
\leq\frac{C}{\left\vert x-y\right\vert ^{2-\epsilon}}\text{ for }\left\vert
x\right\vert >0.
\end{align*}
Denote by $\left(  x_{1}^{\left(  \delta\right)  }\left(  t\right)
,\dots,x_{N}^{\left(  \delta\right)  }\left(  t\right)  \right)  $ the unique
solution of
\[
\frac{dx_{i}^{\left(\delta\right)}\left(t\right)}{dt}=\frac{1}{\sqrt{N}}\sum_
{j\neq i}\xi_{j}K^{\left(  \delta\right) }\left(  x_{i}^{\left(  \delta\right)
}\left(  t\right)  -x_{j}^{\left(  \delta\right)  }\left(  t\right)  \right)
,\qquad i=1,\dots,N
\]
with given (arbitrary)\ initial condition.

\begin{lemma}
Consider the smooth map $x^{0}\mapsto x_{i}^{\left(  \delta\right)  }\left(
t|x^{0}\right)  $ in $\left(  \mathbb{T}^{2}\right)  ^{N}$. Then, the 
probability product measure $Leb_{2N}$ on $\left(  \mathbb{T}^{2}\right)^{N}$ 
is invariant for this map.
\end{lemma}

\begin{proof}
It is a known fact for smooth flows that the determinant is given by the
exponential of the divergence of the vector field, which here is zero; hence,
the determinant is identically equal to one and the flow is Lebesgue measure
preserving. Let us only check that the divergence is zero:\ it is the sum of
divergences on each component $\mathbb{T}^{2}$ of $\left(  \mathbb{T}^{2}
\right)^{N}$, which are all equal to zero because the components have the 
form $\nabla^{\perp}\varphi^{\left(\delta\right)  }\left(  x\right)  $ 
(apply Schwarz theorem on mixed second derivatives).
\end{proof}

Similarly to above, let us introduce the function
\[
L^{\left(  \delta\right)  }\left(  x_{1},\dots,x_{N}\right)  =-\sum
_{\substack{i,j=1,\dots,N \\i\neq j}}\left(  G^{\left(  \delta\right)  }\left(
x_{i}-x_{j}\right)  -k\right)
\]
where $k$ is such that $-\left(  G^{\left(  \delta\right)  }\left(  x\right)
-k\right)  \geq0$ for all $x\in(\mathbb{T}^{2})^{N}$.\newline Setting
$L^{\left(  \delta\right)  }\left(  t\right)  =L^{\left(  \delta\right)
}\left(  x_{1}^{\left(  \delta\right)  }\left(  t\right)  ,\dots,x_{N}^{\left(
\delta\right)  }\left(  t\right)  \right)  $, we have%
\begin{align*}
\frac{dL^{\left(  \delta\right)  }\left(  t\right)  }{dt}  &  =-\sum
_{\substack{i,j=1,\dots,N \\i\neq j}}\nabla G^{\left(  \delta\right)  }\left(
x_{i}^{\left(  \delta\right)  }\left(  t\right)  -x_{j}^{\left(
\delta\right)  }\left(  t\right)  \right)  \cdot\\
&\cdot\frac{1}{\sqrt{N}}\left(\sum_{k\neq i}\xi_{k}\nabla^{\perp}G^{\left(\delta
\right)}\left(  x_{i}^{\left(  \delta\right)  }\left(  t\right)  -x_{k}^{\left(
\delta\right)  }\left(  t\right)  \right)  -\sum_{k\neq j}\xi_{k}\nabla
^{\perp}G^{\left(  \delta\right)  }\left(  x_{j}^{\left(  \delta\right)
}\left(  t\right)  -x_{k}^{\left(  \delta\right)  }\left(  t\right)  \right)
\right)  .
\end{align*}
Again the terms in the last two sums of the form $\nabla^{\perp}G^{\left(
\delta\right)  }\left(  x_{i}^{\left(  \delta\right)  }\left(  t\right)
-x_{j}^{\left(  \delta\right)  }\left(  t\right)  \right)  $ cancel with
$\nabla G^{\left(  \delta\right)  }\left(  x_{i}^{\left(  \delta\right)
}\left(  t\right)  -x_{j}^{\left(  \delta\right)  }\left(  t\right)  \right)
$. Using this estimate and the invariance of Lebesgue measure we can prove:

\begin{lemma}
There exists a constant $C>0$ such that, for all $\delta\in\left(  0,1\right)
$,
\[
-\int_{\left(  \mathbb{T}^{2}\right)  ^{N}}\sup_{t\in\left[  0,T\right]  }%
\sum_{\substack{i,j=1,\dots,N \\i\neq j}}\left(  G^{\left(  \delta\right)
}\left(  x_{i}^{\left(  \delta\right)  }\left(  t|x^{0}\right)  -x_{j}%
^{\left(  \delta\right)  }\left(  t|x^{0}\right)  \right)  -k\right)
dx^{0}\leq C.
\]

\end{lemma}

\begin{proof}
We may summarize the last identity above in the form%
\begin{align*}
\frac{dL^{\left(\delta\right) }\left(  t\right) }{dt} &=\frac{1}{\sqrt{N}}\sum
_{\substack{i,j,k=1,\dots,N\\i\neq j,i\neq k,j\neq k}}a_{ijk}J_{i,j,k}^{\left(
\delta\right)  }\left(  x^{\left(  \delta\right)  }\left(  t\right)  \right)
\\
J_{i,j,k}^{\left(  \delta\right)  }\left(  x\right)   &  :=\nabla G^{\left(
\delta\right)  }\left(  x_{i}-x_{j}\right)  \cdot\nabla^{\perp}G^{\left(
\delta\right)  }\left(  x_{i}-x_{k}\right)
\end{align*}
for suitable coefficients $a_{ijk}$. Integrating in time, taking the supremum
in time over $\left[  0,T\right]  $ and then integrating with respect to the
initial condition $x^{0}$, we have%
\begin{align*}
\int_{\left(  \mathbb{T}^{2}\right)  ^{N}}\sup_{t\in\left[  0,T\right]
}L^{\left(  \delta\right)  }\left(  t\right)  dx^{0}\leq &  \int_{\left(
\mathbb{T}^{2}\right)  ^{N}}\left\vert L^{\left(  \delta\right)  }\left(
0\right)  \right\vert dx^{0}\\
& +\frac{C}{\sqrt{N}}\sum_{\substack{i,j,k=1,\dots,N\\i\neq j,i\neq k,j\neq k}}
\int_{0}^{T}\left(  \int_{\left(  \mathbb{T}^{2}\right)  ^{N}}\left\vert
J_{i,j,k}^{\left(  \delta\right)  }\left(  x^{\left(  \delta\right)  }\left(
s|x^{0}\right)  \right)  \right\vert dx^{0}\right)  ds.
\end{align*}
Now we use the most essential ingredient: the invariance of Lebesgue measure
under the map $x^{0}\mapsto x^{\left(  \delta\right)  }\left(  s|x^{0}\right)
$. This gives us%
\[
\int_{\left(  \mathbb{T}^{2}\right)  ^{N}}\left\vert J_{i,j,k}^{\left(
\delta\right)  }\left(  x^{\left(  \delta\right)  }\left(  s|x^{0}\right)
\right)  \right\vert dx^{0}=\int_{\left(  \mathbb{T}^{2}\right)  ^{N}%
}\left\vert J_{i,j,k}^{\left(  \delta\right)  }\left(  x^{0}\right)
\right\vert dx^{0}.
\]
From the properties imposed on $G^{\left(  \delta\right)  }$ we have%
\[
\left\vert J_{i,j,k}^{\left(  \delta\right)  }\left(  x^{0}\right)  \right\vert
\leq\frac{C}{\left\vert x_{i}-x_{j}\right\vert ^{2-\epsilon}\left\vert
x_{i}-x_{k}\right\vert ^{2-\epsilon}}%
\]
for some constant $C>0$; hence,%
\[
\int_{\left(  \mathbb{T}^{2}\right)  ^{N}}\left\vert J_{i,j,k}^{\left(
\delta\right)  }\left(  x^{0}\right)  \right\vert dx^{0}\leq C\int_{\left(
\mathbb{T}^{2}\right)  ^{2}}\frac{1}{\left\vert x_{i}-x_{k}\right\vert
^{2-\epsilon}}\left(  \int_{\mathbb{T}^{2}}\frac{1}{\left\vert x_{i}%
-x_{j}\right\vert ^{2-\epsilon}}dx_{j}\right)  dx_{i}dx_{k}\leq C^{\prime}%
\]
for some constant $C^{\prime}>0$. Similarly, $\int_{\left(  \mathbb{T}%
^{2}\right)  ^{N}}\left\vert L^{\left(  \delta\right)  }\left(  0\right)
\right\vert dx^{0}\leq C^{\prime\prime}$ for some constant $C^{\prime}>0$. 
In conclusion,
\[
-\int_{\left(  \mathbb{T}^{2}\right)  ^{N}}\sup_{t\in\left[  0,T\right]  }%
\sum_{\substack{i,j=1,\dots,N\\i\neq j}}\left(  G^{\left(  \delta\right)
}\left(  x_{i}^{\left(  \delta\right)  }\left(  t|x^{0}\right)  -x_{j}%
^{\left(  \delta\right)  }\left(  t|x^{0}\right)  \right)  -k\right)
dx^{0}\leq C
\]
for some constant $C>0$.
\end{proof}

\begin{remark}
Without the cancellation of the most singular terms, we would have also%
\[
\int_{\left(  \mathbb{T}^{2}\right)  ^{2}}\frac{1}{\left\vert x_{i}%
-x_{j}\right\vert ^{4-4\epsilon}}dx_{i}dx_{j}=+\infty
\]
(for $\epsilon$ small).
\end{remark}

\begin{theorem}
\label{th:nocollision} For Lebesgue a.e. initial condition $x^{0}\in\Delta
_{N}^{c}$, there is no collision and the solution is global and unique.
\end{theorem}

\begin{proof}
Denote by $d_{T}^{\left(  \delta\right)  }\left(  x^{0}\right)  $ the minimal
distance between vortices of the smoothed system, starting from $x^{0}$, over
$\left[  0,T\right]  $. Then%
\begin{align*}
d_{T}^{\left(  \delta\right)  }\left(  x^{0}\right)    <\delta 
\Longleftrightarrow & \exists t\in\left[  0,T\right]  ,\exists i\neq j:\left\vert
x_{i}^{\left(  \delta\right)  }\left(  t|x^{0}\right)  -x_{j}^{\left(
\delta\right)  }\left(  t|x^{0}\right)  \right\vert <\delta\\
  \Longrightarrow & -\sup_{t\in\left[  0,T\right]  }\sum
_{\substack{i,j=1,\dots,N\\i\neq j}}\left(  G^{\left(  \delta\right)  }\left(
x_{i}^{\left(  \delta\right)  }\left(  t|x^{0}\right)  -x_{j}^{\left(
\delta\right)  }\left(  t|x^{0}\right)  \right)  -k\right) \\
 &  >C\left(  1+\frac{1}{\delta^{1-\epsilon}}\right);
\end{align*}
hence%
\begin{align*}
&  Leb_{2N}\left\{  x^{0}\in\left(  \mathbb{T}^{2}\right)  ^{N}:d_{T}^{\left(
\delta\right)  }\left(  x^{0}\right)  <\delta\right\} \\
&  \leq Leb_{2N}\Bigg\{ x^{0}\in\left(  \mathbb{T}^{2}\right)  ^{N}:\\
&  \qquad\qquad\qquad-\sup_{t\in\left[  0,T\right]  } \sum
_{\substack{i,j=1,\dots,N\\i\neq j}}\left(  G^{\left(  \delta\right)  }\left(
x_{i}^{\left(  \delta\right)  }\left(  t|x^{0}\right)  -x_{j}^{\left(
\delta\right)  }\left(  t|x^{0}\right)  \right)  -k\right)  >C\left(
1+\frac{1}{\delta^{1-\epsilon}}\right)  \Bigg\}\\
&  \leq\frac{C}{C\left(  1+\frac{1}{\delta^{1-\epsilon}}\right)  }%
\end{align*}
where we have used Chebyshev inequality and the lemma and have assumed
$\delta$ so small that $C\left(  1+\frac{1}{\delta^{1-\epsilon}}\right)  >0$.
Thus, for a very large (in the sense of Lebesgue measure) set of initial
conditions $d_{T}^{\left(  \delta\right)  }\left(  x^{0}\right)  \geq\delta$,
which means that $x^{\left(  \delta\right)  }\left(  t|x^{0}\right)  =x\left(
t|x^{0}\right)  $ and no collision occurs. This property is true for a.e.
initial conditions, by the arbitrarity of $\delta$.
\end{proof}

\begin{lemma}
The map $x^{0}\mapsto x\left(  t|x^{0}\right)  $ defined a.e. by the previous
theorem is measurable and preserves Lebesgue measure.
\end{lemma}

\begin{proof}
The claim follows directly from the fact that $x\left(  t|x^{0}\right)
=x^{\left(  \delta\right)  }\left(  t|x^{0}\right)  $ for some sufficiently
small $\delta>0$ depending on $x^{0}$ and $t$ and that the Lebesgue measure is
invariant under the smooth map $x^{0}\mapsto x^{\left(  \delta\right)
}\left(  t|x^{0}\right)$.
\end{proof}

A direct consequence of the previous lemma is that if we consider the initial
condition as a random variable the process defined by Theorem
\ref{th:nocollision} is stationary. Altogether, we obtain the following
theorem on the existence of the vortex dynamics.

\begin{theorem}
\label{th:exvortex} For every $\left(  \xi_{1},\dots,\xi_{N}\right)
\in\mathbb{R}^{N}$ and for $Leb_{2N}$-a.e. $X_{0}^{N}\in\Delta_{N}^{c}$, there
is a unique solution $X_{\cdot}^{N}:[0,\infty)\to\Delta_{N}^{c}$ to
(\ref{eq:odevortexdynamics}).\newline Considering the initial condition as a
random variable distributed as $Leb_{2N}$ the stochastic process $\left(
X_{t}^{N}\right)  _{t}$ is stationary with invariant marginal law $Leb_{2N}$.
\end{theorem}

\section{Construction of a white noise solution}

Now let us consider an i.i.d. sequence $\left(  \xi_{n}\right)  $ of $N(0,1)$
random variables for the intensities and an i.i.d. sequence $\left(  X_{0}%
^{n}\right)  $ of uniformly distributed $\mathbb{T}^{2}$-valued random
variables for the initial positions, both on a probability space
$(\Xi,\mathcal{F},\mathbb{P})$ and independent of each other. We denote by
\begin{align*}
\lambda^{0}_{N}:=(\otimes_{N} N(0,1)) \otimes\; Leb_{2N}%
\end{align*}
the law of the random vector $\left(  (\xi_{1},X_{0}^{1}),\dots,(\xi_{N}%
,X_{0}^{N})\right)  $. Then, by Theorem \ref{th:exvortex} we have that for
a.e. value of $\left(  (\xi_{1},X_{0}^{1}),\dots,(\xi_{N},X_{0}^{N})\right)  $
there is a unique solution $X_{\cdot}^{N} :[0,\infty)\to\Delta_{N}^{c}$ to
(\ref{eq:odevortexdynamics}). We are interested in a measure-valued vorticity
field $\theta_{N}$ which is a linear combination of the processes $X_{\cdot
}^{N}$ weighted with the intensities. For this field we obtain results similar
to those for the white noise and it will be the basis to construct a white
noise solution to mSQG.

\begin{theorem}
\label{th:propthetan} Let the intensities and initial positions $(\xi
,X_{0}):=\left(  (\xi_{1},X_{0}^{1}),\dots, (\xi_{N},X_{0}^{N})\right)  $ be
distributed as $\lambda_{N}^{0}$ and $(X_{\cdot}^{1},\dots, X_{\cdot}%
^{N}):(0,\infty)\to\Delta_{N}^{c}$ be the corresponding dynamics. Then we have
for the measure-valued field
\begin{align*}
\theta^{N}_{t}(\xi,X_{0}):=\frac{1}{\sqrt N}\sum_{n=1}^{N} \xi_{n}%
\delta_{X_{t}^{n}(\xi,X_{0})}%
\end{align*}

\begin{itemize}
\item[i)] $\theta^{N}_{t}$ is stationary and space homogeneous.

\item[ii)] $\partial_{t} \left<  \theta_{t}^{N},\phi\right>  =\left<
\theta_{t}^{N} \otimes\theta_{t}^{N},H_{\phi}\right>  $ for all $t>0$ and
$\phi\in C^{\infty}(\mathbb{T}^{2})$.

\item[iii)] $\mathbb{E}\left[  \left<  \theta_{t}^{N}\otimes\theta_{t}%
^{N},H_{\phi}\right>  ^{2} \right]  \leq c \|H_{\phi}\|_{L^{2}(\mathbb{T}%
^{2}\times\mathbb{T}^{2})}^{2}$ for all $t>0$ and $\phi\in C^{\infty
}(\mathbb{T}^{2})$.

\item[iv)] $\mathbb{E}\left[  \|\theta_{t}^{N}\|^{p}_{H^{-1-\delta}%
(\mathbb{T}^{2})} \right]  \leq c_{p,\delta}$ for all $p\geq1, \delta>0$
independent of $t$ and $N$.

\item[v)] $\theta_{t}^{N} \overset{law}{\rightharpoonup}\theta_{WN}$ for all
$t>0$, where $\theta_{WN}$ denotes white noise and convergence takes place in
$H^{-1-\delta}(\mathbb{T}^{2})$ for all $\delta>0$.
\end{itemize}
\end{theorem}

\begin{proof}
We only show ii) and iii) here; for the other points, we refer to
\cite{Flandoli}. By definition of $\theta_{t}^{N}$ we have for $t>0$ and
$\phi\in C^{\infty}(\mathbb{T}^{2})$
\begin{align*}
\partial_{t} \left<  \theta_{t}^{N},\phi\right>   &  = \partial_{t}
\frac{1}{\sqrt N}\sum_{n=1}^{N}\xi_{n} \phi(X_{t}^{n}) = 
\frac{1}{\sqrt N}\sum_{n=1}^{N}\xi_{n}(\partial_{t} X_{t}^{n})
\nabla\phi(X_{t}^{n})\\
&  = \frac{1}{N}\sum_{n=1}^{N}\sum_{i=1}^{N}\xi_{n}\xi_{i} K(X_{t}^{n}%
-X_{t}^{i}) \nabla\phi(X_{t}^{n})\\
&  = \frac{1}{2N}\sum_{n=1}^{N}\sum_{i=1}^{N}\xi_{n}\xi_{i} K(X_{t}^{n}%
-X_{t}^{i}) (\nabla\phi(X_{t}^{n})-\nabla\phi(X_{t}^{i}))\\
&  = \frac{1}{N}\sum_{n=1}^{N}\sum_{i=1}^{N}\xi_{n}\xi_{i} H_{\phi}(X_{t}%
^{n},X_{t}^{i}) = \frac{1}{N}\sum_{n=1}^{N}\sum_{i=1}^{N}\xi_{n}\xi_{i}
\left<  \delta_{X_{t}^{n}} \otimes\delta_{X_{t}^{i}}, H_{\phi}\right> \\
&  =\left<  \theta_{t}^{N} \otimes\theta_{t}^{N},H_{\phi}\right>  ,
\end{align*}
which is the equality ii). This equality also yields
\begin{align*}
\left<  \theta_{t}^{N} \otimes\theta_{t}^{N},H_{\phi}\right>  ^{2} = \frac
{1}{N^{2}}\sum_{n=1}^{N}\sum_{i=1}^{N}\sum_{m=1}^{N}\sum_{j=1}^{N} \xi_{n}%
\xi_{i}\xi_{m}\xi_{j} \; H_{\phi}(X_{t}^{n},X_{t}^{i}) H_{\phi}(X_{t}%
^{m},X_{t}^{j})
\end{align*}
and taking the expectation, we deduce from the independence of the random
variables
\begin{align*}
\mathbb{E}\left[  \left<  \theta_{t}^{N} \otimes\theta_{t}^{N},H_{\phi
}\right>  ^{2}\right]  = \frac{1}{N^{2}}\sum_{n=1}^{N}\sum_{i=1}^{N}\sum
_{m=1}^{N}\sum_{j=1}^{N} \mathbb{E}\left[  \xi_{n}\xi_{i}\xi_{m}\xi
_{j}\right]  \; \mathbb{E}\left[  H_{\phi}(X_{t}^{n},X_{t}^{i}) H_{\phi}%
(X_{t}^{m},X_{t}^{j})\right]  .
\end{align*}
On the one hand we have $\mathbb{E}\left[  \xi_{n}\right]  =0$ ($1\leq n \leq
N$) and the $\xi_{n}$ are independent, so only terms with $n=i$ and $m=j$,
$n=m$ and $j=i$ or $n=j$ and $m=i$ contribute to the sum and on the other hand
we have $H_{\phi}(x,x)=0$, so we can also exclude the case $n=i$ and $m=j$.
Altogether,
\begin{align*}
\mathbb{E}\left[  \left<  \theta_{t}^{N} \otimes\theta_{t}^{N},H_{\phi
}\right>  ^{2}\right]  =  &  \frac{1}{N^{2}}\sum_{n=1}^{N}\sum_{i=1}^{N}
\mathbb{E}\left[  \xi_{n}^{2}\xi_{i}^{2}\right]  \; \mathbb{E}\left[  H_{\phi
}(X_{t}^{n},X_{t}^{i})^{2}\right] \\
&  +\frac{1}{N^{2}}\sum_{n=1}^{N}\sum_{i=1}^{N} \mathbb{E}\left[  \xi_{n}%
^{2}\xi_{i}^{2}\right]  \; \mathbb{E}\left[  H_{\phi}(X_{t}^{n},X_{t}^{i})
H_{\phi}(X_{t}^{i},X_{t}^{n})\right] \\
=  &  \frac{2}{N^{2}}\sum_{n=1}^{N}\sum_{i=1}^{N} \mathbb{E}\left[  \xi
_{n}^{2}\xi_{i}^{2}\right]  \; \mathbb{E}\left[  H_{\phi}(X_{t}^{n},X_{t}%
^{i})^{2}\right] \\
\leq &  c \sup_{1\leq n,i\leq N} \mathbb{E}\left[  H_{\phi}(X_{t}^{n}%
,X_{t}^{i})^{2}\right] \\
=  &  c \|H_{\phi}\|_{L^{2}(\mathbb{T}^{2}\times\mathbb{T}^{2})}^{2},
\end{align*}
where we used the invariance of the Lebesgue measure under the process
$X_{t}^{n}$ once again.
\end{proof}\ \\

Part i) of the theorem also implies that the law of $\left(  (\xi_{1}%
,X_{t}^{1}),\dots, (\xi_{N},X_{t}^{N})\right)  $ is $\lambda_{N}^{0}$ for all
$t$. The process $\theta_{\cdot}^{N}$ is not used directly to construct a
white noise solution, but by Skorokhod's representation theorem we can find
processes with the same law which converge to a white noise solution. To do
so, we first have to prove tightness.

\begin{lemma}
\label{lemma:qnconvweakly} Let $\theta_{\cdot}^{N}$ be as in Theorem
\ref{th:propthetan} and $Q^{N}$ denote its law. Then the sequence $(Q_{N}%
)_{N}$ is tight in $C^{1}([0,T];H^{-1-}(\mathbb{T}^{2}))$ and converges weakly
to some probability measure $Q$ in this space.
\end{lemma}

\begin{proof}
First we show that $(Q_{N})_{N}$ is bounded in probability in $L^{p}%
((0,T);H^{-1-\delta} (\mathbb{T}^{2}))\cap W^{1,2}((0,T);H^{-\gamma}%
(\mathbb{T}^{2}))$ for any $p\geq1$, $\delta>0$, $\gamma>3$. Estimating the
$L^{p}$-part is easy; we have by Theorem \ref{th:propthetan} and the
stationarity of $\theta_{t}^{N}$
\begin{align*}
\mathbb{E}\left[  \int_{0}^{T}\|\theta_{t}^{N}\|^{p}_{H^{-1-\delta}%
(\mathbb{T}^{2})} dt \right]   &  = \int_{0}^{T}\mathbb{E}\left[  \|\theta
_{t}^{N}\|^{p}_{H^{-1-\delta}(\mathbb{T}^{2})} \right]  dt \leq c_{p,\delta}T.
\end{align*}
For the $W^{1,2}$-norm we replace $-1-\delta$ by $-\gamma$ and take $p=2$
which gives
\begin{align*}
\mathbb{E}\left[  \int_{0}^{T}\|\theta_{t}^{N}\|^{2}_{H^{-\gamma}%
(\mathbb{T}^{2})} dt \right]   &  \leq c_{2,\gamma}T.
\end{align*}
To estimate the derivative in time let $\phi\in C^{\infty}(\mathbb{T}^{2})$.
We have
\begin{align*}
\partial_{t} \left<  \theta_{t}^{N},\phi\right>  =\left<  \theta_{t}%
^{N}\otimes\theta_{t}^{N},H_{\phi} \right>
\end{align*}
which implies by using Theorem \ref{th:propthetan} again
\begin{align*}
\mathbb{E}\left[  |\partial_{t} \left<  \theta_{t}^{N},\phi\right>  |^{2}
\right]   &  = \mathbb{E} \left[  |\left<  \theta_{t}^{N}\otimes\theta_{t}%
^{N},H_{\phi}\right>  | \right]  \leq C \|H_{\phi}\|^{2}_{L^{2}(\mathbb{T}%
^{2}\times\mathbb{T}^{2})}.
\end{align*}
Choosing $\phi=e_{k}$ with $e_{k}(x)=e^{2ikx}$ ($k\in\mathbb{Z}^{2},
x\in\mathbb{T}^{2}$), we obtain
\begin{align*}
|H_{e_{k}}(x,y)|\leq c |k| \frac{\left| e_{k}(x)-e_{k}(y)\right|}{|x-y|
^{2-\epsilon}} \leq c|k|^{2} \frac{1}{|x-y|^{1-\epsilon}}%
\end{align*}
and thus, $\|H_{e_{k}}\|_{L^{2}(\mathbb{T}^{2}\times\mathbb{T}^{2})}\leq c
k^{2}$. This yields
\begin{align*}
\mathbb{E}\left[  \int_{0}^{T}\|\partial_{t}\theta_{t}^{N}\|^{2}_{H^{-\gamma
}(\mathbb{T}^{2})}dt \right]   &  = \mathbb{E}\left[  \int_{0}^{T}\sum
_{k\in\mathbb{Z}^{2}}(1+|k|^{2})^{-\gamma}| \left<  \partial_{t}\theta_{t}%
^{N},e_{k}\right>  |^{2}dt \right] \\
&  \leq cT \sum_{k\in\mathbb{Z}^{2}}(1+|k|^{2})^{-\gamma}|k|^{4}%
\end{align*}
and the sum is bounded due to $\gamma>3$. So we have that $(Q_{N})_{N}$ is
bounded in probability in $L^{\infty-}((0,T);H^{-1-}(\mathbb{T}^{2}))\cap
W^{1,2}((0,T);H^{-\gamma} (\mathbb{T}^{2}))$ and this space embeds compactly in
$C([0,T],H^{-1-} (\mathbb{T}^{2}))$ (see \cite{Flandoli}, Corollary 27).
Applying now Prokhorov's theorem proves the weak convergence to some
probability measure $Q$.
\end{proof}\ \\

Finally, we have everything in place to prove Theorems \ref{theorem:existence} 
and \ref{theorem:approximation}. For the reader's convenience, let us repeat 
their statements at this point.

\begin{theorem}\label{th:mainresult}
For all $\epsilon\in(0,1)$, there exists a probability space $(\Xi,\mathcal{F},P)$ 
and a measurable map $\theta_{\cdot}:\Xi\times[0,T]\to C^{\infty}(\mathbb{T}^{2})
^{\prime}$ ($T>0$ arbitrary) such that

\begin{itemize}
\item[i)] $\theta_{\cdot}$ is a white noise solution to mSQG and $(\theta
_{t})$ is stationary.

\item[ii)] The random point vortex system defined on $(\Xi,\mathcal{F},P)$
has a subsequence which converges P-a.s. to $\theta$ in $C([0,T];H^{-1-}%
(\mathbb{T}^{2}))$.
\end{itemize}
\end{theorem}

\begin{proof}
From Lemma \ref{lemma:qnconvweakly} and Skorokhods representation theorem, it
follows that there is a probability space $(\overline{\Xi}, \overline
{\mathcal{F}},\overline{P})$ and $C([0,T];H^{-1-}(\mathbb{T}^{2}))$-valued
random variables $\overline{\theta_{\cdot}}^{N_{k}}$ with law $Q_{N_{k}}$
converging to some $\overline{\theta_{\cdot}}$ in $C([0,T];H^{-1-}%
(\mathbb{T}^{2}))$ ($Q_{N_{k}}$ a subsequence of $Q_{N}$). By the method
described in \cite{Flandoli}, Lemma 28 we can extend $(\overline{\Xi
},\overline{\mathcal{F}},\overline{P})$ to some probability space
$(\Xi,\mathcal{F},P)$, define new processes $\theta^{N_{k}}$ with law
$Q_{N_{k}}$ converging to a process $\theta_{\cdot}$ in $C([0,T];H^{-1-}%
(\mathbb{T}^{2}))$, and we have the representation
\begin{align*}
\theta_{t}^{N_{k}}=\frac{1}{\sqrt{N_{k}}}\sum_{n=1}^{N_{k}}\xi_{n}%
\delta_{X_{t}^{n,N_{k}}},
\end{align*}
where $(X_{\cdot}^{1,N_{k}},\dots,X_{\cdot}^{N_{k},N_{k}})$ is the solution to
(\ref{eq:odevortexdynamics}) for a new random vector (without renaming)
$\left(  (\xi_{1},X^{1,N_{k}}),\dots,(\xi_{N},X^{N_{k},N_{k}})\right)  $ with
law $\lambda_{N}^{0}$. What we have to show is that $\theta_{\cdot}$ is a
white noise solution in the sense of Definition \ref{def:wnsolution}. We know
already that $\theta_{t}$ is a white noise for all $t$ and from
\begin{align*}
\left<  \theta_{t}^{N_{k}},\phi\right>  -\left<  \theta_{0}^{N_{k}}%
,\phi\right>  -\int_{0}^{t}\left<  \theta_{s}^{N_{k}}\otimes\theta_{s}^{N_{k}%
},H_{\phi}\right>  ds=0
\end{align*}
follows
\begin{align*}
&  \mathbb{E} \left[  \left|  \left<  \theta_{t},\phi\right>  -\left<
\theta_{0},\phi\right>  -\int_{0}^{t}\left<  \theta_{s}\otimes\theta
_{s},H_{\phi}\right>  ds \right|  \wedge1 \right] \\
&  \leq\mathbb{E} \left[  \left|  \left<  \theta_{t}-\theta_{t}^{N_{k}}%
-\theta_{0}+\theta_{0}^{N_{k}}, \phi\right>  \right|  \wedge1\right]
+\mathbb{E} \left[  \int_{0}^{t} \left|  \left<  (\theta_{s} \otimes\theta
_{s})-(\theta^{N_{k}}_{s}\otimes\theta^{N_{k}}_{s}),H_{\phi}\right>  \right|
ds \wedge1 \right]  .
\end{align*}
The convergence of $\theta^{N_{k}}_{\cdot}$ to $\theta_{\cdot}$ in
$C([0,T];H^{-1-} (\mathbb{T}^{2}))$ implies
\begin{align*}
\mathbb{E} \left[  \left|  \left<  \theta_{t}-\theta_{t}^{N_{k}},\phi\right>
- \left<  \theta_{0}-\theta_{0}^{N_{k}},\phi\right>  \right|  \wedge1\right]
\to0 \quad(N_{k}\to\infty).
\end{align*}
For the nonlinear part, let $(H^{m}_{\phi})_{m}$ be a smooth approximation of
$H_{\phi}$ such that $H^{m}_{\phi}(x,y)=H^{m}_{\phi}(y,x)$ and $H^{m}_{\phi
}(x,x)=0$ for $x,y \in\mathbb{T}^{2}$. Then
\begin{align*}
\mathbb{E} \left[  \left|  \int_{0}^{t}\left<  \theta_{s}\otimes\theta
_{s},H_{\phi}-H^{m}_{\phi}\right>  ds\right|  \wedge1 \right]   &  \leq
\int_{0}^{t} \mathbb{E}\left[  \left|  \left<  \theta_{s}\otimes\theta
_{s},H_{\phi}-H^{m}_{\phi} \right>  ds\right|  \right] \\
&  \leq\int_{0}^{t} \mathbb{E}\left[  \left|  \left<  \theta_{s}\otimes
\theta_{s},H_{\phi}-H^{m}_{\phi} \right>  ds\right|  ^{2} \right]  ^{1/2}\\
&  \to0 \quad(m\to\infty)
\end{align*}
by our definition of $\left<  \theta_{s}\otimes\theta_{s},H_{\phi}\right>  $
and
\begin{align*}
\mathbb{E} \left[  \left|  \int_{0}^{t}\left<  \theta^{N_{k}}_{s}\otimes
\theta^{N_{k}}_{s}, H_{\phi}-H^{m}_{\phi}\right>  ds\right|  \wedge1 \right]
&  \leq\int_{0}^{t} \mathbb{E}\left[  \left|  \left<  \theta_{s}^{N_{k}%
}\otimes\theta_{s}^{N_{k}},H_{\phi}- H^{m}_{\phi}\right>  ds\right|  ^{2}
\right]  ^{1/2}\\
&  \leq c \| H_{\phi}-H^{m}_{\phi} \|_{L^{2}(\mathbb{T}^{2}\times
\mathbb{T}^{2})}\\
&  \to0 \quad(m\to\infty)
\end{align*}
independent of $N_{k}$. Thus,
\begin{align*}
\mathbb{E} \left[  \int_{0}^{t} \left|  \left<  (\theta_{s} \otimes\theta
_{s})-(\theta^{N_{k}}_{s}\otimes\theta^{N_{k}}_{s}),H_{\phi}\right>  \right|
ds \wedge1 \right]  \to0 \quad(N_{k}\to\infty)
\end{align*}
and this gives
\begin{align*}
\mathbb{E} \left[  \left|  \left<  \theta_{t},\phi\right>  -\left<  \theta
_{0},\phi\right>  -\int_{0}^{t}\left<  \theta_{s}\otimes\theta_{s},H_{\phi
}\right>  ds \right|  \wedge1 \right]  =0,
\end{align*}
i.e., $\left<  \theta_{t},\phi\right>  -\left<  \theta_{0},\phi\right>
-\int_{0}^{t}\left<  \theta_{s}\otimes\theta_{s},H_{\phi}\right>  ds=0$ P-a.s.
for all $\phi\in C^{\infty} (\mathbb{T}^{2})$.
\end{proof}

\end{document}